\documentclass[11pt,a4paper]{article}
\usepackage[utf8]{inputenc}
\usepackage{amsmath, amssymb, mathrsfs, amsthm, amsfonts}
\usepackage{latexsym}
\usepackage{a4wide}

\newtheorem{thm}{Theorem}[section]   

\newtheorem{lemma}[thm]{Lemma}

\newtheorem{example}[thm]{Example}
\newtheorem{defn}[thm]{Definition}
\newtheorem{rem}[thm]{Remark}

\def\NN{{\mathbb N}}
\def\RR{{\mathbb R}}
\def\QQ{{\mathbb Q}}

\def\A{{\mathcal A}}

\def\U{{\mathcal U}}

\def\aa{{\alpha}}

\author{F. Patrick Rabarison\footnote{University of Antananarivo,
 Madagascar, \texttt{prabarison@gmail.com}},  Hery Randriamaro\footnote{International Centre for Theoretical Physics, Italy,
\texttt{hrandria@ictp.it}}}

\title{Generalized Number Systems and Application to Hyperoctahedral Groups}

\begin{document}

\maketitle

\begin{abstract}
In this work, we generalize the integer enumeration basis. We also construct bijections
between the elements of special sets and the elements of some groups and treat the special case of hyperoctohedral groups. We then find an analogous of the Lehmer code for the hyperoctahedral groups.
\end{abstract}

\tableofcontents

\section{Introduction and notation}

A number system is a framework for representing numerals. Classical numbering system or enumeration 
system with base $q$  ($q\in \NN$) are expressing all natural number $ n\in \NN $ in the form:

\begin{equation}\label{classical}
 n=\sum_{i=0}^{l(n)} n_i. q^i \; , \text{ where } n_i \in \{0,1,2,\cdots,q-1 \}
\end{equation}

Such representation are for example used in number theory or combinatorics as usefull tool for finding exotic congruences. In other hand one have also the factorial number system which is already known in $19th$ century by Cantor \cite{cantor}. This is expressing
each $ n\in \NN_0 $ in the form

\begin{equation}\label{classical}
 n=\sum_{i=0}^{l} f_i . i! \; , \text{ where } f_i \in \{0,1,2,\cdots,i \},
\end{equation}
for some $l\in \NN_0$, and it is  well known that these representation are unique. 
\noindent
Known as Lehmer code \cite{lehmer}, Laisant \cite{laisant} build a code by associating an element of this representation an element of a symmetric group. This assiciation is  proven to be a bijection between the set $\{0,1,\cdots,n!-1 \}$ and the symmetric group $S_n$. 

This article is organised as follow:  in section 2, we will begin by generalizing these notion of representations. For that we given
the condition for a infinite sequence of numbers to be an enumeration basis. We will also extend the results to 
rationals numbers and real numbers.  In section 3,  we apply this methode for some sequence of positive integrers and find a Lehmer code analogous result for the hyperoctahedral groups.

\subsection*{Notation}
We make the convention of notation:
\begin{itemize}
 \item $\NN=\{1,2,3,\cdots \}$ is the set of non-negative integers.
 \item $\NN_0=\NN \cup \{0\}$.
 \item $\aa=(\aa_0, \aa_1, \aa_2,\cdots)$, $\U=(U_0,U_1,U_2,\cdots )$  are sequences of non-negative integers.
 \item $\beta= 1+\aa$ is a sequence of integers such that $\beta_i=1+\aa_i $ for all $i\in \NN_0$.
 \item Bold symbols $\mathbf{1},\mathbf{2},\mathbf{i},\mathbf{j},\mathbf{k}$ will denote respectively some standard vectors $e_1, e_2, e_i, e_j, e_k $ of the canonical basis of the vector space $\RR^n$.
 \item $\mathsf{B} = (\mathsf{B_n})_{\mathsf{n} \in \NN_0} = (2^nn!)_{n\in \NN_0}$ is a sequence of integers.
 \item $(B_n)_{n \in \NN}$ is a sequence of hyperoctahedral groups and should not be confounded with the sequence $\mathsf{B}$.
 \item For a numbering system $(\U,\aa)$ (see Definition \ref{basis}), $<U_0, \dots, U_n>$ will denote the 
 set $\{\sum_{i=0}^n a_i U_i\ |\ 0 \leq a_i \leq \alpha_i\}$.
\end{itemize}

\section{General basis for positive integers}
\subsection{General Number systems}
Let us start with two sequences of integers $\aa=(\aa_n)_{n\in \NN}$ and  $\U=(U_n)_{n\in \NN}$.
We want to express all positive integer $n$ with an expression:
\begin{equation}\label{general}
 n=\sum_{i=0}^{l} k_i. U^i \; , \text{ where } k_i \in \{0,1,2,\cdots,\aa_i \} \text{ and for some } l\in \NN_0
\end{equation}

\begin{defn}\label{basis}
 Let $\U$ and $\aa$ be two sequences of non-negative integers. We say that the system $(\U,\aa)$ is a system of enumeration or a number system if and only if
all elements of $\NN_0$ can be represented as in (\ref{general}) and this representation is unique.
\end{defn}
\noindent
It is important here that $U_n$ and $\aa_n$ are non-negative integers. Other authors considers for example some $U_n$ negative or even complex
 numbers. See for example 
\cite[Chapter~4]{knuth2} for this topic . The unicity is also important in this definition. The representation as in (\ref{general}) are also called 
radix system but the representation may not be unique. That means that for a suitable fixed sequence $\aa=(\aa_n)_{n\in \NN}$, we need to find the appropriate sequence $\U=(\U_n)_{n\in \NN}$ such that
the representation above is unique. And conversely, for a fixed sequence $\U=(\U_n)_{n\in \NN_0}$, the question is: are
there any appropriate sequence $\aa=(\aa_n)_{n\in \NN}$ such that all positive integers are representable with the system $(\U,\aa)$ ? 
\paragraph{Convention.}
For a fixed number system $(\U,\aa)$, and $a_n\in \NN_0$ with $a_n\leq\aa_n$, we use then the convention 
$${a_{n}a_{n-1}a_{n-2}\cdots a_{2}a_{1}a_{0}} \text{~ or ~ } {a_{n}:a_{n-1}:a_{n-2}:\cdots a_{2}:a_{1}:a_{0}}$$ 
to denote the number $\sum_{i=0}^na_iU_i$.

\begin{lemma}\label{divrem}
Consider a fixed number system $(\U,\aa)$ and let  $a\in \NN$ with $$a=a_na_{n-1}a_{n-2}\cdots a_{2}a_{1}a_{0}$$ with $a_n\neq 0$, then
$$a_nU_n\leq a \leq (a_n+1)U_n$$
\end{lemma}
\begin{proof}
 By Euclidean division, one can write $a=a_nU_n+r$ with $r=a_{n-1}\cdots a_{2}a_{1}a_{0}\le U_n$ .
\end{proof}

\begin{lemma}\label{unicity}
If $(\U,\aa)$ is a number system and 
$$a=a_na_{n-1}a_{n-2}\cdots a_{2}a_{1}a_{0}=b_mb_{m-1}m_{m-2}\cdots b_{2}b_{1}b_{0}$$
with $a_n\neq 0$ and $b_m\neq 0$, then
$$n=m \text{ and }  a_i=b_i \text{~ for all ~} i \in \{0,1,2,\cdots n\} \; .$$
\end{lemma}
\begin{proof}
By the Lemma \ref{divrem}, the relation $n=m$ is immediate because if $m > n$ then $a$ will be smaller than $U_m$  which is absurd. So $n=m$ and one have the unicity 
 by induction  and by the unicity of the expression $a=a_nU_n+r_n$.
\end{proof}

\begin{thm}\label{firsttheorem}
Let  $\aa=(\aa_n)_{n\in \NN}$ and $\U=(U_n)_{n\in \NN}$ be two sequences of positive integers. Then, $(\U,\aa)$ is an number system if and only if
$\forall n\in \NN_0, \aa_n\geq 1$ , $U_0=1$ and
\begin{equation}\label{alphaun}
U_n=\prod_{i=0}^{n-1}({1+\aa_i})=\prod_{i=0}^{n-1}{\beta_i} \; .
\end{equation}
i.e all natural number can be expressed as in (\ref{general}) and the representation is unique.
\end{thm}
\begin{proof}
By construction and the Lemma \ref{unicity}, we just need that 
$1+\sum_{i=0}^{n} \aa_i. U^i $ should be equal to $U_n$. The last condition means that 
$\frac{U_{n}}{U_{n-1}}=(1+\aa_{i-1})$ which lead to the formula $U_n=\prod_{i=0}^{n-1}({1+\aa_i})$.
\end{proof}
The theorem above also means that once, we fix a specific  sequence $\aa=(\aa_n)_{n\in \NN}$ with $\aa_n\geq 1$, then there is only one sequence $\U=(U_n)_{n\in \NN}$
such that $(\U,\aa)$ is a number system. And conversely, for a fiven positive sequence of integer $\U=(U_n)_{n\in \NN}$, there will we a corresponding
$\aa=(\aa_n)_{n\in \NN}$ such that $(\U,\aa)$ is a number system if $U_0=1$ and $\forall n\in \NN, \beta_n:=\frac{U{n+1}}{U_n} \in \NN$.

\paragraph{A Horner's-like procedure.} Horner's scheme is based on expressing a polynomial by a particular expression
 so that the value of the corresponding polynomial function is quickly
obtained. For instance:
$$
    a + bx  +c x^2 d x^3 + e x^4 \;=\; a + x \cdot\left(b + x \cdot(c + x \cdot(d + x \cdot e))\right) 
$$
For a number system $(\U,\aa)$, recall that $U_n=\prod_{i=0}^{n-1}\beta_i$. To express a positive integer $a$ in the $(\U,\aa)$-system, one proceed
to the following: Begin by dividing $a$ by $\beta_0$ and take $a_0$ to be the rest $a_0:=r_0 $ :
$$a=r_0+\beta_0q_0 \; .$$
Divide $q_1$ by $\beta_1$ and take the $a_1$ to be the rest $a_1:=r_1 $:
$$q_1=r_1+\beta_1q_2 \; .$$
Continue the procedure untill $q_n=0$ for some $n\in \NN_0$, by dividing $q_i$ by $\beta_i$ and take $a_i:=r_i$:
$$q_i=r_i+\beta_iq_{i+1} \; .$$
Clearly, we have:
$a=\; a_0 + \beta_0 \cdot\left(a_1 + \beta_1 \cdot(a_2 + \beta_2 \cdot(a_3 + \cdots))\right) $.

\begin{table}
\begin{center}
  \begin{tabular}{|| c | c | c || c | c | c ||}
    \hline
    Classical & Factorial & Hyperoctahedral  & Classical  & Factorial &  Hyperoctahedral \\ \hline
    0 & 0 & 0  & 40 & 1220 & 500       \\ \hline
    1 & 1 & 1   & 41 & 1221 & 501   \\ \hline
    2 & 10 & 10 & 42 & 1300 & 510    \\ \hline
    3 & 11 & 11  & 43 & 1301 & 511   \\ \hline
    4 & 20 & 20   & 44 & 1310 & 520   \\ \hline
    5 & 21 & 21   & 45 & 1311 & 521   \\ \hline
    6 & 100& 30   & 46 & 1320 & 530   \\ \hline
    7 & 101 & 31   & 47 & 1321 & 531  \\ \hline
    8 & 110 & 100  & 48 & 2000 & 1000  \\ \hline
    9 & 111 & 101  & 49 & 2001 & 1001  \\ \hline
    10 & 120 & 110  & 50 & 2010 & 1010 \\ \hline
    11 & 121 & 111  & 51 & 2011 & 1011 \\ \hline
    12 & 200 & 120  & 52 & 2020 & 1020\\ \hline
    13 & 201 & 121 & 53 & 2021 & 1021\\ \hline
    14 & 210 & 130 & 54 & 2100 & 1030 \\ \hline
    15 & 211 & 131  & 55 & 2101 & 1031\\ \hline
    16 & 220 & 200 & 56 & 2110 &  1100\\ \hline
    17 & 221 & 201  & 57 & 2111 & 1101\\ \hline
    18 & 300 & 210  & 58 & 2120 & 1110\\ \hline
    19 & 301 & 211 & 59 & 2121 & 1111 \\ \hline
    20 & 310 & 220  & 60 & 2300 & 1120\\ \hline
    21 & 311 & 221  & 61 & 2301 & 1121\\ \hline
    22 & 320 & 230  & 62 & 2310 & 1130\\ \hline
    23 & 321 & 231  & 63 & 2311 & 1131\\ \hline
    24 & 1000 & 300 & 64 & 2320 & 1200 \\ \hline
    25 & 1001 & 301 & 65 & 2321 & 1201 \\ \hline
    26 & 1010 & 310  & 66 & 3000 & 1210\\ \hline
    27 & 1011 & 311 & 67 & 3001 & 1211 \\ \hline
    28 & 1020 & 320  & 68 & 3010 & 1220\\ \hline
    29 & 1021 & 321  & 69 & 3011 & 1221\\ \hline
    30 & 1100 & 330 & 70 & 3020 & 1230\\ \hline
    31 & 1101 & 331  & 71 & 3021 & 1231\\ \hline
    32 & 1110 & 400  & 72 & 3100 & 1300 \\ \hline
    33 & 1111 & 401  & 73 & 3101 & 1301\\ \hline
    34 & 1120 & 410 & 74 & 3110 & 1310 \\ \hline
    35 & 1121 & 411  & 75 & 3111 & 1311 \\ \hline
    36 & 1200 & 420 & 76 & 3120 & 1320 \\ \hline
    37 & 1201 & 421  & 77 & 3121 & 1321\\ \hline
    38 & 1210 & 430 & 78 & 3200 & 1330 \\ \hline
    39 & 1211 & 431  & 79 & 3201 & 1331\\ \hline
    
  \end{tabular}
\caption{Representing positive integers in the Classical system $\U^{(10)}$, in the Factorial system and in the  Hyperoctahedral system.}\label{table1}
\end{center}
\end{table}

\paragraph{Classical examples.} If one choose $m \in \NN$ suth that $\aa_i=m-1$ with $m \geq 2$, then we recover the classical system $(\U,\aa)$ with 
$$\U=\U^{(m)}=(1,m,m^2,m^3,\cdots)$$
and the addition and multiplication procedures are well known.

If now  we choose $A:=(A_n)_n$ with $A_n=(n+1)!$ $\aa_A=(n+1)_n$ for all $n\in \NN_0$, then one obtains  
\begin{defn}[Factorial System]\label{factorialsystem}
One define the factorial system to be the number system:
\begin{equation}
 (A,\aa_{A}):= (((n+1)!)_{n\in\NN_0},(n+1)_{n\in \NN_0}) \; .
\end{equation}
\end{defn}

\subsection{The hyperoctahedral number system}
Let us now take as basis $\mathsf{B}=(\mathsf{B_0}, \mathsf{B_1}, \mathsf{B_2}, \mathsf{B_3}, \dots)$ and the corresponding $\aa$ such that
\begin{equation}\label{basebn}
 \mathsf{B_n}=2^nn! \;  \text{ corresponds to } \aa_n=2n+1 \; .
\end{equation}
This means that $\aa=(1,3,5,7,\dots)$ and the first twenty numbers are represented in table (\ref{table1}).
\begin{defn}[Hyperoctahedral System]\label{hyperoctahedralnumbers}
 One define the hyperoctahedral system to be the system:
\begin{equation}
 (\mathsf{B},\, \aa_{\mathsf{B}}):= ((2^nn!)_{n\in\NN_0},\, (2n+1)_{n\in\NN_0}) \; .
\end{equation}
\end{defn}
This definition  is motivated by the fact that we choose $U_n=\mathsf{B_n}$ to be the cardinal of the hyperoctahedral group $B_n$. We have:
\begin{thm}
Every positive integer a has a unique representation in the hyperoctahedral system i.e $(\mathsf{B},\aa_{\mathsf{B}})$ is a number system.
\end{thm}
\begin{proof}
 This is obvious since by the Theorem \ref{firsttheorem}, one just need to verify that $\frac{\mathsf{B_{n+1}}}{\mathsf{B_n}}=2n+2$.
\end{proof}

\subsection{Extension to rational and real numbers}
One can extends the represenatation to all integers by adding the opposite of all positive integers. But one can also extend it to representation of 
rational numbers and real numbers.

\begin{defn}\label{rationalplus}
Let $(\U,\aa)$ a number system, one can extend the system  with $U_{-n}=\frac{1}{U_n}$ and $\aa_{-n}=\aa_n$. Then,  define
the set of number $\A_\aa^+$ to be the set of real numbers of the form
\begin{equation}
 \sum_{-\infty}^m a_iU_i 
\end{equation}
and 
\begin{equation}
\A_\aa= \A_\aa^+ \cup \A_\aa^- \text{ where }  \A_\aa^-=-\A_\aa^+
\end{equation}
\end{defn}
\noindent
Note that for this definition holds, it is necessary to verify that the serie is convergent. This exercice is left to the reader. 
\begin{thm}
Let $(\U,\aa)$ an extended number system, then all rational integer can be represented in this system i.e:
\begin{equation}
\QQ \subseteq A_\aa \; .
\end{equation}
\end{thm}
\begin{proof}
 This is an easy exercice for the reader.
\end{proof}
\begin{rem}
Unfortunately, unlike the Theorem \ref{firsttheorem}, the representation is not unique for rational numbers. For instance, for a fixed
number system $(\U,\aa)$, then one has $U_{n+1}=U_n(1+\aa_n)$ which lead to the ralation:
\begin{equation}\label{keyun}
\frac{\aa_n}{U_{n+1}}= \frac{1}{U_n}-\frac{1}{U_{n+1}} \; .
\end{equation}
From that, one deduce easily that 
\begin{equation}\label{doublerepresentation}
 \frac{1}{U_n}=\sum_{i=n}^\infty \frac{\aa_n}{U_{n+1}}
\end{equation}
\end{rem}
For example, the factoradic representation of a rational number $\tfrac{a}{b}$ with $\gcd(a,b)=1$ in the 
open unit interval, i.e. $0 < \tfrac{a}{b} < 1$, is defined as
   $$ \frac{a}{b} := \sum_{i=1}^{N} \frac{d_i}{(i+1)!},\quad 0 \le d_i \le i, \, $$
where $d_i,\, 0 \le d_i \le i$, is the "factoradic digit" for place-value $\frac{1}{(i+1)!}$, and $N$ is the number 
of "factoradic digits" after the "factoradic point".

The remark above also show for instance that rational numbers may have multiple factoradic representations:
\begin{equation}
    \frac{1}{m!} = \sum_{i=m}^{\infty} \frac{i}{(i+1)!},\quad m \ge 1, \, 
\end{equation}
where on the left side we have the terminating form, while on the right we have the nonterminating form. This is analogous to
\begin{equation}
    \frac{1}{b^m} = \sum_{i=m}^{\infty} \frac{b-1}{b^{i+1}},\quad m \ge 0,\, b \ge 2, \, 
\end{equation}
for the number system $(\U^{(b)},(b-1))_{n\in\NN}$.
\paragraph{Representation of rational in an extended number system.}
One can  consider without loss of generality that the rational number $\frac{p}{q}\in ]0,1[$. And write:
\begin{equation}
 \frac{p}{q}=\frac{b_0}{\beta_0}+\frac{b_1}{\beta_0\beta_1}+\frac{b_2}{\beta_0\beta_1\beta_2}+\frac{b_3}{\beta_0\beta_1\beta_2\beta_3}+\cdots \; .
\end{equation}
Then instead of dividing as in the Horner's like procedure, one multiply by the $\beta_i$ and take the integer parts:
$$b_0=\lfloor \frac{p}{q} \beta_0 \rfloor, b_1=\lfloor (\frac{p}{q} \beta_0 -b_0)\beta_1 \rfloor, 
b_2=\lfloor ((\frac{p}{q} \beta_0 -b_0)\beta_1)-b_1)\beta2 \rfloor, \cdots \; .$$
Note that this porcedure also lead to the development of a real number in a certain basis.
\paragraph{Examples.} Some rationals and  real numbers  with a factoradic representation:\\
    $e = 1:0.1:1:1:1:1:1:1:1:1:1:... $\\
    $\frac{23}{24} = 0:1.2:3 $\\
     $\frac{47}{30} = 1:1.0:1:3 $
\paragraph{Examples.} Some rationals and real numbers with a hyperoctahedral representation:\\
$\sqrt{e} = 1.1:1:1:1:1::1:1:... $\\
$\frac{13}{16} = 0.1:2:3$\\
$\frac{205}{69} = 2.0:1:4$
\section{Applications to some Coxeter groups}
\subsection{The Lehmer code}

Cantor seems to be the first to introduce the factorial number system in \cite{cantor}. Then, Laisant (cf \cite{laisant}) introduce a code by associating  
each element of the symmetric group  $S_n$ to a positive integer in ${0,1,2,,\cdots,n!-1}$. He prove that this association is in fact a 
bijection. The key ingredient is using a statistic on $S_n$, namely the inversion statistic.

If a permutation $\sigma$  is specified by the sequence $(\sigma_1, \cdots, \sigma_n)$ of its images of $1, \cdots, n$, then it
 is encoded by a sequence of n numbers, but not all such sequences are valid since every number must be used only once. 

The Lehmer code is the sequence
$$
    L(\sigma)=(L(\sigma)_1,\ldots,L(\sigma)_n)\quad\text{where}\quad L(\sigma)_i=\#\{ j>i : \sigma_j<\sigma_i \},
$$
in other words the term $L(\sigma )_i$ counts the number of terms in $(\sigma_1, \cdots, \sigma_n)$ to the right of $\sigma_i$ that are
 smaller than it, a number between 0 and $n - i$, allowing for $n + 1 - i$ different values.

A pair of indices $(i,j)$ with $i < j$ and $\sigma_i > \sigma_j$ is called an inversion of $\sigma$, and $L(\sigma )_i$ counts the number
 of inversions $(i,j)$ with i fixed and varying j. It follows that 
$$L(\sigma )_1 + L(\sigma )_2 + \cdots + L(\sigma )_n$$ 
is the total number of inversions of $\sigma$ , which is also
 the number of adjacent transpositions that are needed to transform the permutation into the identity permutation. 

In this section, we will be investigating on costruction of an analogue of this code for the Hyperoctahedral groups $B_n$. The key ingredients of our construction are based on the existence of inversion on this family of groups.

\subsection{A code related to hyperoctahedral groups}

\noindent Recall that the hyperoctahedral group $B_n$ is the group of signed permutations of the coordinates in $\mathbb{R}^n$. We 
write $\mathbf{i}$ for the $i$th standard basis vector. This interaction between simple font for an integer and bold font for the
 corresponding standard vector is kept in all this section. For instance, for a
signed permutation $\pi$, $\pi(\mathbf{i})$ refers to the $|\pi(i)|$th standard basis vector. We use the notation
$$\pi = \left( \begin{array}{cccc} \mathbf{1} & \mathbf{2} & \dots & \mathbf{n}\\ 
\pi(\mathbf{1}) & \pi(\mathbf{2}) & \dots & \pi(\mathbf{n}) \end{array} \right)$$
for an element $\pi$ of $B_n$ with $\pi(\mathbf{i}) \in \{\pm\mathbf{1}, \dots, \pm\mathbf{n}\}$. The element $\pi$ is an 
invertible linear map of $\mathbb{R}^n$.\\
It is convenient to think $B_n$ as the Coxeter group of the same type with root system
$$\Phi_n = \{\pm\mathbf{i},\,\pm\mathbf{i} \pm\mathbf{j} \ |\ 1 \leq i \neq j \leq n\},$$
and positive root system $$\Phi_n^+ = \{\mathbf{k},\,\mathbf{i} + \mathbf{j},\,\mathbf{i} - \mathbf{j}\ |\ k \in [n],\, 1 \leq i < j \leq n\}.$$
We use the number of inversions defined by Reiner \cite[2.~Preliminaries]{reiner} that is: The number of inversions of the element $\pi$ of $B_n$ is
$$\mathtt{inv}\,\pi = \# \{v \in \Phi_n^+ \ |\ \pi^{-1}(v) \in -\Phi_n^+\}.$$
Let us consider the following subset of $\Phi_n^+$ defined by
\begin{equation} \label{eq1}
\Phi_{n,i}^+ = \{\mathbf{i},\,\mathbf{i} + \mathbf{j},\,\mathbf{i} - \mathbf{j}\ |\ i < j \leq n\},
\end{equation}
and define the number of $i$-inversions of $\pi$ by
\begin{equation} \label{eq2}
\mathtt{inv}_i\,\pi = \# \{v \in \Phi_{n,i}^+ \ |\ \pi^{-1}(v) \in -\Phi_n^+\}.
\end{equation} 

\begin{example} \label{B2element}
In the following table, we see the corresponding $1$-inversions and $2$-inversions of the eight elements of $B_2$:
\begin{table}
\begin{center}
\begin{tabular}{ l | l }
$\pi_0 = \left( \begin{array}{cc} \mathbf{1} & \mathbf{2} \\ \mathbf{1} & \mathbf{2} \end{array} \right)$ &
$\mathtt{inv}_1(\pi_0)\ \mathtt{inv}_2(\pi_0) = 0 \ 0$ \\
$\pi_1 = \left( \begin{array}{cc} \mathbf{1} & \mathbf{2} \\ \mathbf{1} & -\mathbf{2} \end{array} \right)$ &
$\mathtt{inv}_1(\pi_1)\ \mathtt{inv}_2(\pi_1) = 0 \ 1$ \\
$\pi_2 = \left( \begin{array}{cc} \mathbf{1} & \mathbf{2} \\ \mathbf{2} & \mathbf{1} \end{array} \right)$ &
$\mathtt{inv}_1(\pi_2)\ \mathtt{inv}_2(\pi_2) = 1 \ 0$ \\
$\pi_3 = \left( \begin{array}{cc} \mathbf{1} & \mathbf{2} \\ \mathbf{2} & -\mathbf{1} \end{array} \right)$ &
$\mathtt{inv}_1(\pi_3)\ \mathtt{inv}_2(\pi_3) = 1 \ 1$ \\
$\pi_4 = \left( \begin{array}{cc} \mathbf{1} & \mathbf{2} \\ -\mathbf{2} & \mathbf{1} \end{array} \right)$ &
$\mathtt{inv}_1(\pi_4)\ \mathtt{inv}_2(\pi_4) = 2 \ 0$ \\
$\pi_5 = \left( \begin{array}{cc} \mathbf{1} & \mathbf{2} \\ -\mathbf{2} & -\mathbf{1} \end{array} \right)$ &
$\mathtt{inv}_1(\pi_5)\ \mathtt{inv}_2(\pi_5) = 2 \ 1$ \\
$\pi_6 = \left( \begin{array}{cc} \mathbf{1} & \mathbf{2} \\ -\mathbf{1} & \mathbf{2} \end{array} \right)$ &
$\mathtt{inv}_1(\pi_6)\ \mathtt{inv}_2(\pi_6) = 3 \ 0$ \\
$\pi_7 = \left( \begin{array}{cc} \mathbf{1} & \mathbf{2} \\ -\mathbf{1} & -\mathbf{2} \end{array} \right)$ &
$\mathtt{inv}_1(\pi_7)\ \mathtt{inv}_2(\pi_7) = 3 \ 1$
\end{tabular}
\end{center}
\caption{The $1$-inversions and $2$-inversions of the elements of $B_2$.} \label{B2table}
\end{table}
\end{example}

\begin{lemma} \label{lemequa}
Let $i \in [n]$ and $\pi \in B_n$. If
\begin{itemize}
\item[$(i)$] $\pi(\mathbf{i}) = \mathbf{j}$, then
$$\mathtt{inv}_i(\pi) = \# \big\{k \in \{i+1, \dots, n\} \ |\ j > |\pi(k)|\big\},$$
\item[$(ii)$] $\pi(\mathbf{i}) = -\mathbf{j}$, then
$$\mathtt{inv}_i(\pi) = 1 + \# \big\{k \in \{i+1, \dots, n\} \ |\ j > |\pi(k)|\big\} + 2.\# \big\{k \in \{i+1, \dots, n\} \ |\ j < |\pi(k)|\big\}.$$
\end{itemize}
\end{lemma}

\begin{proof}
Use the definition of the number of $i$-inversions (Equation \ref{eq1} and Equation \ref{eq2}).
\end{proof}

\begin{lemma} \label{lemone}
Let $\sigma, \tau \in B_n$. If $\sigma(\mathbf{1}) \neq \tau(\mathbf{1})$ then $\mathtt{inv}_1\,\sigma \neq \mathtt{inv}_1\,\tau$.
\end{lemma}

\begin{proof}
We use the equations of Lemma \ref{lemequa}. Let $i,j \in [n]$. If
\begin{itemize}
\item[$\bullet$] $\sigma(\mathbf{1}) = \mathbf{i}$ and $\tau(\mathbf{1}) = \mathbf{j}$ with $i<j$, then
$\mathtt{inv}_1\,\tau - \mathtt{inv}_1\,\sigma = j-i,$
\item[$\bullet$] $\sigma(\mathbf{1}) = -\mathbf{i}$ and $\tau(\mathbf{1}) = -\mathbf{j}$ with $i<j$, then
$\mathtt{inv}_1\,\sigma - \mathtt{inv}_1\,\tau = j-i,$
\item[$\bullet$] $\sigma(\mathbf{1}) = \mathbf{i}$ and $\tau(\mathbf{1}) = -\mathbf{j}$, then
$\mathtt{inv}_1\,\sigma < n \leq \mathtt{inv}_1\,\tau.$
\end{itemize}
\end{proof}

\begin{lemma} \label{value1}
We have $\mathtt{inv}_1\,B_n = \{0, \dots, 2n-1\}$. 
\end{lemma}

\begin{proof}
On one hand, we deduce from Lemma \ref{lemequa} that the minimal value of $\mathtt{inv}_1$ is $0$, corresponding to the elements $\pi \in B_n$ such that $\pi(\mathbf{1})= \mathbf{1}$, and the maximal is $2n-1$, corresponding to the elements $\pi \in B_n$ such that $\pi(\mathbf{1})= -\mathbf{1}$.\\
One the other hand, we deduce from Lemma \ref{lemone} that $\mathtt{inv}_1$ has $2n$ possible values. 
\end{proof}

\noindent Let us consider the enumeration basis $\mathsf{B} = (\mathsf{B_0}, \mathsf{B_1}, \mathsf{B_2}, \mathsf{B_3}, \dots)$ relative to the hyperoctahedral groups. Recall that $\mathsf{B_n}=2^nn!$ with corresponding maximal positive 
integer $\alpha_n = 2n+1$. The aim of this section is to prove the following theorem.

\begin{thm} \label{B2thm}
There is a one-to-one correspondence between the elements of $B_n$ and those of $<\mathsf{B_0}, \mathsf{B_1}, \dots, \mathsf{B_{n-1}}>$ with respect to the $i$-inversions. This bijection is given by
$$b_n: \left. \begin{array}{ccc} B_n & \rightarrow & <\mathsf{B_0}, \dots, \mathsf{B_{n-1}}> \\ 
\pi & \mapsto & \sum_{i=1}^n \mathtt{inv}_{n-i+1}(\pi) \mathsf{B_{i-1}} \end{array} \right..$$
\end{thm}

\begin{proof}
For $n=1$, we have trivially $b_1(\left( \begin{array}{c} \mathbf{1} \\ \mathbf{1} \end{array} \right)) = 0$ and
 $b_1(\left( \begin{array}{c} \mathbf{1} \\ \mathbf{-1} \end{array} \right))= \mathsf{B_0}$. From 
table \ref{B2table} of Example \ref{B2element}, we can deduce $b_2$. For instance
$b_2(\left( \begin{array}{cc} \mathbf{1} & \mathbf{2} \\ -\mathbf{2} & -\mathbf{1} \end{array} \right)) = 2\mathsf{B_1} + \mathsf{B_0}$.\\
Now, we assume that the bijection exists for any positive integer smaller than $n$ and prove by induction that $b_{n+1}$ also exists.\\
For $\pi \in B_{n+1}$ and $i \in [n+1]$, we write $\mathtt{sg}\,\pi(\mathbf{i})$ for the sign of $\pi(\mathbf{i})$.\\
Let $\pi = \left( \begin{array}{cccc} \mathbf{1} & \mathbf{2} & \dots & \mathbf{n+1}\\ 
\pi(\mathbf{1}) & \pi(\mathbf{2}) & \dots & \pi(\mathbf{n+1}) \end{array} \right) \in B_{n+1}$. We consider the $n$ rightmost columns of $\pi$, i.e. $\hat{\pi} = \left( \begin{array}{ccc} \mathbf{2} & \dots & \mathbf{n+1}\\ 
\pi(\mathbf{2}) & \dots & \pi(\mathbf{n+1}) \end{array} \right)$, and define $\overline{\pi} =
 \left( \begin{array}{ccc} \mathbf{2} & \dots & \mathbf{n+1}\\ 
\overline{\pi(\mathbf{2})} & \dots & \overline{\pi(\mathbf{n+1})} \end{array} \right)$ with
$$\overline{\pi(\mathbf{i})} = \left\{ \begin{array}{ll} \pi(\mathbf{i)}+\, \mathtt{sg} \big(\pi(\mathbf{i})\big) \, \mathbf{1} & \text{if}\ |\pi(i)| < |\pi(1)|,\\ \pi(\mathbf{i}) & \text{if}\ |\pi(i)| > |\pi(1)|.\end{array} \right.$$
Then, we get an element $\overline{\pi}$ of $B_n$ and we have
\begin{equation} \label{hatover}
\sum_{i=1}^n \mathtt{inv}_{n-i+2}(\hat{\pi}) \mathsf{B_{i-1}} = b_n(\overline{\pi}).
\end{equation}
We can now proceed of the proof of the existence of the bijection $b_{n+1}$.\\
Let $\sigma, \tau \in B_n$ such that $\sigma \neq \tau$. If
\begin{itemize}
\item[$\bullet$] $\sigma(\mathbf{1}) \neq \tau(\mathbf{1})$, then, from Lemma \ref{lemone}, we have
$\mathtt{inv}_{1}(\sigma)\mathsf{B_n} \neq \mathtt{inv}_{1}(\tau)\mathsf{B_n}$,
\item[$\bullet$] $\sigma(\mathbf{1}) = \tau(\mathbf{1})$, then, from Equation \ref{hatover}, we get
$$\sum_{i=1}^n \mathtt{inv}_{n-i+2}(\sigma) \mathsf{B_{i-1}} \neq \sum_{i=1}^n \mathtt{inv}_{n-i+2}(\tau) \mathsf{B_{i-1}}.$$
\end{itemize}
So, on one hand, the map $\pi \mapsto \sum_{i=1}^{n+1} \mathtt{inv}_{n-i+2}(\pi) \mathsf{B_{i-1}}$ is injective.\\ 
On the other hand, we deduce from Lemma \ref{value1} that $\mathtt{inv}_1\,B_{n+1} = \{0, \dots, \alpha_n\}$.
\end{proof}

\bibliographystyle{abbrvnat}

\begin{thebibliography}{50} 

\bibitem{bourbaki} Bourbaki, Nicolas (1968), ``Groupes et algèbres de Lie'', Ch. 4–6, Eléments
 de Mathématique, Fasc. XXXIV, Hermann, Paris; Masson, Paris, 1981.

\bibitem{bjorner-brento} Bjorner, Anders - Brenti, Francesco ``Combinatorics of Coxeter Groups'', Graduate Texts in Mathematics, vol. 231, Springer, 2005. 

\bibitem{cantor} Cantor, G. (1869),  Über einfache Zahlensysteme, Zeitschrift für Mathematik und Physik 14 (1869), 121–128.
 

\bibitem{knuth2} Knuth, D. E. (1997), ``Volume 2: Seminumerical Algorithms '', The Art of Computer
 Programming (3rd ed.), Addison-Wesley, p. 192, ISBN 0-201-89684-2.

\bibitem{laisant} Laisant, Charles-Ange (1888), ``Sur la numération factorielle, application aux permutations '', Bulletin de la Société
 Mathématique de France  16: 176–183.


\bibitem{lehmer} Lehmer, D.H. (1960), ``Teaching combinatorial tricks to a
 computer'', Proc. Sympos. Appl. Math. Combinatorial Analysis, Amer. Math. Soc. 10: 179–193

\bibitem{reiner} Reiner, Victor (1993), ``Signed Permutation Statistics'', Eur. J. Comb. 14(6): 553-567.







\end{thebibliography}

\end{document}